\documentclass[12pt,twoside]{amsart}
\usepackage{amsmath,amsthm,amssymb}
\usepackage{mathrsfs,amsthm,amscd}
\usepackage[all]{xy}

\title{On semipositivity theorems}
\author{Osamu Fujino and Taro Fujisawa}
\date{2017/11/10, version 0.23}
\address{Department of Mathematics, Graduate School of Science,
Osaka University, Toyonaka, Osaka 560-0043, Japan}
\email{fujino@math.sci.osaka-u.ac.jp}
\address{Department of Mathematics, School of Engineering, 
Tokyo Denki University, Tokyo, Japan}
\email{fujisawa@mail.dendai.ac.jp}
\newcommand{\Gr}[0]{\operatorname{Gr}}
\newtheorem{thm}{Theorem}[section]
\newtheorem{lem}[thm]{Lemma}
\newtheorem{prop}[thm]{Proposition}
\newtheorem{cor}[thm]{Corollary}

\theoremstyle{definition}
\newtheorem{defn}[thm]{Definition}
\newtheorem{rem}[thm]{Remark}
\newtheorem*{ack}{Acknowledgments}
\newtheorem{say}[thm]{}

\makeatletter
    
    \@addtoreset{equation}{section}
\makeatother
\begin{document}
\bibliographystyle{amsalpha+}

\maketitle

\begin{abstract}
We generalize the Fujita--Zucker--Kawamata semipositivity 
theorem from the analytic viewpoint. 
\end{abstract}

\tableofcontents

\section{Introduction}\label{f-sec1}

The main purpose of this paper is 
to generalize the well-known 
Fujita--Zucker--Kawamata semipositivity 
theorem (see \cite[\S 4.~Semi-positivity]{kawamata1}, 
\cite[Theorem 2]{kawamata2}, 
\cite[Section 5]{fujino-fujisawa}, 
\cite[Theorem 3]{fujino-fujisawa-saito}, and 
\cite{fujisawa}) 
from the analytic viewpoint. 

\begin{thm}\label{f-thm1.1}
Let $X$ be a complex manifold and 
let $X_0\subset X$ be 
a Zariski open set such that $D=
X\setminus X_0$ is a normal crossing divisor on $X$. 
Let $V_0$ be a polarizable variation of $\mathbb R$-Hodge 
structure over $X_0$ with 
unipotent monodromies around $D$. 
Let $F^b$ be the canonical extension 
of the lowest piece of the Hodge 
filtration. 
Let $F^b\to \mathscr L$ be a quotient line bundle of $F^b$. 
Then the Hodge metric of $F^b$ induces 
a singular hermitian metric $h$ on $\mathscr L$ 
such that $\sqrt{-1}\Theta_h(\mathscr L)\geq 0$ and 
the Lelong number of $h$ is zero everywhere. 
\end{thm}

As a direct consequence of Theorem \ref{f-thm1.1}, we have: 

\begin{cor}[cf.~\cite{kawamata3}]\label{f-cor1.2}
Let $X$ be a complex manifold and 
let $X_0\subset X$ be a Zariski open set such that $D=
X\setminus X_0$ is a normal crossing divisor on $X$. 
Let $V_0$ be a polarizable 
variation of $\mathbb R$-Hodge structure over $X_0$ with 
unipotent monodromies around $D$. 
Let $F^b$ be the canonical extension 
of the lowest piece of the Hodge 
filtration. 
Then $\mathscr O_{\mathbb P_X(F^b)}(1)$ has 
a singular hermitian metric $h$ such that 
$\sqrt{-1}\Theta_h(\mathscr O_{\mathbb P_X(F^b)}(1))
\geq 0$ and that the Lelong number of $h$ is 
zero everywhere. Therefore, $F^b$ is nef 
in the usual sense when $X$ is projective. 
\end{cor}

\begin{rem}\label{f-rem1.3}
There exists a quite short published proof of Corollary \ref{f-cor1.2} 
(see the proof of \cite[Theorem 1.1]{kawamata3}). 
However, we have been unable to follow it. 
We also note that the arguments in \cite[\S4.~Semi-positivity]{kawamata1} 
contain various troubles. For the details, see 
\cite[4.6.~Remarks]{fujino-fujisawa-saito}.  
\end{rem}

\begin{rem}\label{f-rem1.4}
When $X$ is projective and $V_0$ is geometric in 
Corollary \ref{f-cor1.2}, the nefness of $F^b$ has already 
played important roles in the Iitaka program and 
the minimal model program for higher-dimensional 
complex algebraic varieties. 
\end{rem}

More generally, we can prove: 

\begin{thm}\label{f-thm1.5}
Let $X$ be a complex manifold and 
let $X_0\subset X$ be 
a Zariski open set such that $D=
X\setminus X_0$ is a normal crossing divisor on $X$. 
Let $V_0$ be a polarizable variation of $\mathbb R$-Hodge 
structure over $X_0$ with 
unipotent monodromies around $D$. 
If $\mathscr M$ is a holomorphic line subbundle of the associated 
system of Hodge bundles $\Gr_F^\bullet \mathscr V=\bigoplus _p \Gr _F^p 
\mathscr V$ which 
is contained in the kernel of the Higgs field 
$$
\theta: \Gr_F ^\bullet\mathscr 
V\to \Omega^1_X(\log D)\otimes_{\mathscr O_{X}} \Gr _F^\bullet \mathscr V, 
$$ then 
the Hodge metric induces 
a singular hermitian metric $h$ on 
its dual $\mathscr M^{\vee}$ such that 
$\sqrt{-1}\Theta_h(\mathscr M^{\vee})\geq 0$ and 
that the Lelong number of $h$ is 
zero everywhere. 
\end{thm}

For the details of the Higgs field $
\theta: \Gr_F ^\bullet\mathscr 
V\to \Omega^1_X(\log D)\otimes_{\mathscr O_{X}} \Gr _F^\bullet \mathscr V$ 
in Theorem \ref{f-thm1.5}, 
see Definition \ref{f-def2.7} below. 

As a direct easy consequence of Theorem \ref{f-thm1.5}, 
we obtain: 

\begin{cor}[{\cite{zuo} and \cite[Theorem 1.8]{brunebarbe}}]\label{f-cor1.6}
Let $X$ be a complex manifold and 
let $X_0\subset X$ be 
a Zariski open set such that $D=
X\setminus X_0$ is a normal crossing divisor on $X$. 
Let $V_0$ be a polarizable variation of $\mathbb R$-Hodge 
structure over $X_0$ with 
unipotent monodromies around $D$. 
If $A$ is a holomorphic subbundle of the associated 
system of Hodge bundles $\Gr_F^\bullet \mathscr V
=\bigoplus _p \Gr _F^p \mathscr V$ which 
is contained in the kernel of the Higgs field 
$$
\theta: \Gr_F^\bullet \mathscr 
V\to \Omega^1_X(\log D)\otimes \Gr _F^\bullet \mathscr V, 
$$ 
then $\mathscr O_{\mathbb P_X(A^{\vee})}(1)$ 
has a singular hermitian metric $h$ such that 
$\sqrt{-1}\Theta _h (\mathscr O_{\mathbb P_X(A^{\vee})}
(1))\geq 0$ and that the Lelong number of 
$h$ is zero everywhere. 
Therefore, the dual vector bundle $A^{\vee}$ is nef in the usual sense when 
$X$ is projective. 
\end{cor}

Corollary \ref{f-cor1.6} is an analytic version 
of \cite[Theorem 1.8]{brunebarbe} (see also \cite{fujisawa}). 
For some generalizations of \cite[Theorem 1.8]{brunebarbe} 
from the Hodge module theoretic viewpoint, 
see \cite[Theorem 18.1]{popa-schnell} and 
\cite[Theorem A]{popa-wu}. 
For a very recent development on semipositivity theorems from 
the theory of Higgs bundles, see \cite{brunebarbe2}. 

\begin{rem}\label{f-rem1.7} 
Let $a$ be the integer such that $F^{a+1}_0\subsetneq F^a_0=\mathscr 
V_0$. 
Then, in Corollary \ref{f-cor1.6}, $\Gr^a_F\mathscr V$ is a holomorphic 
subbundle of $\Gr_F^\bullet\mathscr V$ and is contained in 
the kernel of $\theta$. 
Therefore, we can use Corollary \ref{f-cor1.6} 
for $A=\Gr^a_F\mathscr V$. 
By considering the dual Hodge structure in Corollary \ref{f-cor1.6} 
and putting $A=\Gr^a_F\mathscr V$, 
Corollary \ref{f-cor1.6} is also a generalization of 
the Fujita--Zucker--Kawamata 
semipositivity theorem (see, for example, 
\cite[Remark 3.15]{fujino-fujisawa}). 
Of course, by considering the dual Hodge structure, 
Theorem \ref{f-thm1.5} contains Theorem \ref{f-thm1.1} 
as a special case. 
\end{rem}

Our proof in this paper 
heavily depends on \cite{kollar}, which 
is based on \cite{cks}, and Demailly's approximation result 
for quasi-plurisubharmonic functions on complex manifolds 
(see \cite{demailly1} and \cite{demailly}). 

\begin{rem}[Singular hermitian metrics on vector bundles]\label{f-rem1.8} 
We note that our results explained above are local analytic. 
Therefore, we can easily see that the Hodge metric 
of $F^b$ in Theorem \ref{f-thm1.1} is a semipositively curved singular 
hermitian metric in the sense of P\u aun--Takayama (see \cite[Definition 2.3.1]
{paun-takayama} and \cite[Lemma 18.2]{hps}). 
Moreover, in Corollary \ref{f-cor1.6}, the induced metric on $A$ is a 
seminegatively 
curved singular hermitian metric in the sense of P\u aun--Takayama 
(see \cite[Definition 2.3.1]{paun-takayama} and 
\cite[Lemma 18.2]{hps}). 
For the details of singular hermitian metrics on vector bundles and some related 
topics, 
see \cite{paun-takayama} (see also \cite{hps} and \cite{brunebarbe}). 
\end{rem}

\begin{ack}
The first author was partially supported by JSPS KAKENHI Grant 
Numbers JP2468002, JP16H03925, JP16H06337. 
He thanks Professor Dano Kim whose question made him 
consider the problems discussed in this paper. 
Moreover, he pointed out an ambiguity in a preliminary version of 
this paper.  
The authors thank Professor Shin-ichi Matsumura for 
answering their questions and giving them some useful comments 
on singular hermitian metrics on vector bundles. 
Finally, they also thank the referee for comments. 
\end{ack}

\section{Preliminaries}\label{f-sec2}
In this section, we collect some basic 
definitions and results. 

\begin{say}[Singular hermitian metrics, multiplier ideal sheaves, 
and so on]\label{f-say2.1}
Let us recall some basic definitions and facts about singular hermitian 
metrics and plurisubharmonic functions. For the details, 
see \cite[(1.4), (3.12), (5.4), and so on]{demailly}. 

\begin{defn}[Singular hermitian metrics and curvatures]\label{f-def2.2}
Let $\mathscr L$ be a holomorphic line bundle on a 
complex manifold $X$. 
A {\em{singular hermitian metric}} $h$ on $\mathscr L$ is a metric 
which is given in every trivialization $\theta: \mathscr L|_U
\simeq U\times \mathbb C$ by 
$$
|\!| \xi |\!|_h=|\theta(\xi)|e^{-\varphi(x)}, 
\quad x\in U, \ \xi \in \mathscr L_x, 
$$
where $\varphi\in L^1_{\mathrm{loc}}(U)$ is an arbitrary function, called 
the {\em{weight}} of the metric with 
respect to the trivialization $\theta$. 
Note that $L^1_{\mathrm{loc}}(U)$ is the space of locally 
integrable functions on $U$. 
The {\em{curvature}} $\Theta_h(\mathscr L)$ 
of a singular hermitian metric $h$ on $\mathscr L$ is defined by 
$$
\Theta _h(\mathscr L):=2\partial \overline \partial \varphi, 
$$ 
where $\varphi$ is a weight function and $\partial \overline 
\partial \varphi$ is taken in the sense of 
currents.  
It is easy to see that the right hand side does not 
depend on the choice of trivializations. 
Therefore, we get a global closed $(1, 1)$-current $\Theta_h(\mathscr L)$ 
on $X$. In this paper, $\sqrt{-1}
\Theta_h(\mathscr L)\geq 0$ means that 
$\sqrt{-1}\Theta_h(\mathscr L)$ is positive 
in the sense of currents.  

Let $\mathscr L$ be a holomorphic line bundle on a smooth 
projective variety $X$. 
Then it is well known that there exists a singular hermitian 
metric $h$ on $\mathscr L$ with 
$\sqrt{-1}\Theta_h(\mathscr L)\geq 0$ if and 
only if $\mathscr L$ is pseudoeffective (see 
\cite[(6.17) Theorem (c)]{demailly}).  
\end{defn}

\begin{defn}[(Quasi-)plurisubharmonic functions]\label{f-def2.3}
A function $\varphi:U\to [-\infty, \infty)$ defined on an 
open set $U\subset \mathbb C^n$ is called 
{\em{plurisubharmonic}} if 
\begin{itemize}
\item[(i)] $\varphi$ is upper semicontinuous, and 
\item[(ii)] for every complex line $L\subset \mathbb C^n$, 
$\varphi|_{U\cap L}$ is subharmonic on $U\cap L$, 
that is, for every $a\in U$ and $\xi \in \mathbb C^n$ satisfying 
$|\xi|<d(a, U^c)=\inf \{|a-x| \, | \, x\in U^c\}$, the function 
$\varphi$ satisfies the mean inequality 
$$
\varphi(a)\leq \frac{1}{2\pi}\int ^{2\pi}_0 \varphi(a+e^{i\theta}
\xi)d\theta. 
$$ 
\end{itemize}

Let $X$ be an $n$-dimensional complex manifold. 
A function $\varphi:X\to [-\infty, \infty)$ is said to 
be {\em{plurisubharmonic}} 
if there exists an open cover $X=\bigcup _{i\in I}U_i$ such that 
$\varphi|_{U_i}$ is plurisubharmonic on $U_i$ ($\subset \mathbb C^n$) for every 
$i$. 
A {\em{quasi-plurisubharmonic}} function is a  function $\varphi$ which 
is locally equal to the sum of a plurisubharmonic function and of a smooth 
function. 

Let $\varphi$ be a quasi-plurisubharmonic function on a complex 
manifold $X$. 
Then the {\em{multiplier ideal sheaf}} 
$\mathscr J(\varphi)\subset \mathscr O_X$ is defined by 
$$
\Gamma (U, \mathscr J (\varphi))=\{ f\in 
\mathscr O_X(U) \, |\,  |f|^2e^{-2\varphi}\in L^1_{\mathrm{loc}}(U)\}
$$ 
for every open set $U\subset X$. 
It is well known that $\mathscr J(\varphi)$ is a coherent ideal sheaf on $X$. 
\end{defn}

\begin{defn}[Lelong numbers]\label{f-def2.4}
Let $\varphi$ be a quasi-plurisubharmonic function 
on $U$ ($\subset \mathbb C^n$). 
The Lelong number $\nu(\varphi, x)$ 
of $\varphi$ at $x\in U$ is defined as follows: 
$$
\nu(\varphi, x)=\underset{z\to x}{\liminf} \frac{\varphi(z)}
{\log |z-x|}.  
$$ 
It is well known that $\nu (\varphi, x)\geq 0$. 
\end{defn}

In this paper, we will implicitly use the following easy lemma 
repeatedly. 

\begin{lem}\label{f-lem2.5}
Let $\mathscr L$ be a holomorphic line bundle on a complex manifold $X$. 
Let $h=ge^{-2\varphi}$ be a singular hermitian metric on $\mathscr L$, where 
$g$ is a smooth hermitian metric on $\mathscr L$ and $\varphi$ is a locally 
integrable function on $X$. 
We assume that $\sqrt{-1}\Theta_h(\mathscr L)\geq 0$. 
Then there exists a quasi-plurisubharmonic function $\psi$ on $X$ such that 
$\varphi$ coincides with $\psi$ almost everywhere. 
In this situation, we put $\mathscr J(h)=\mathscr J(\psi)$. 
Moreover, 
we simply say the Lelong number of $h$ to denote 
the Lelong number of $\psi$ if there is no risk of confusion. 
\end{lem}
\end{say}

\begin{say}[Systems of Hodge bundles, 
Higgs fields, curvatures, and so on]\label{f-say2.6}

Let us recall the definition of systems of Hodge bundles. 

\begin{defn}[Systems of Hodge bundles]\label{f-def2.7}
Let $V_0=(\mathbb V_0, F_0)$ be a polarizable variation of 
$\mathbb R$-Hodge structure on a complex manifold $X_0$, 
where $\mathbb V_0$ is a local system of finite-dimensional 
$\mathbb R$-vector spaces on $X_0$ and 
$\{F^p_0\}$ is the Hodge filtration. 
Then we obtain a Higgs bundle $(E_0, \theta_0)$ on $X_0$ by 
setting 
$$
E_0=\Gr_{F_0}^\bullet \mathscr V_0=\bigoplus _p F^p_0/F^{p+1}_0
$$ 
where $\mathscr V_0=\mathbb V_0\otimes \mathscr O_{X_0}$. 
Note that $\theta_0$ is induced by the Griffiths transversality 
$$
\nabla: F^p_0\to \Omega^1_{X_0}\otimes _{\mathscr O_{X_0}}
F^{p-1}_0. 
$$
More precisely, $\nabla$ induces 
$$
\theta^p_0: F^p_0/F^{p+1}_0\to \Omega^1_{X_0}\otimes 
_{\mathscr O_{X_0}}\left (F^{p-1}_0/F^p_0\right)
$$ for every $p$. 
Then 
$$\theta_0=\bigoplus_p \theta^p_0: E_0\to \Omega^1_{X_0}\otimes 
_{\mathscr O_{X_0}}E_0.
$$ 
The pair $(E_0, \theta_0)$ is usually called 
the {\em{system of Hodge bundles}} associated to $V_0=(\mathbb V_0, 
F_0)$ and $\theta_0$ is called the {\em{Higgs field}} of $(E_0, \theta_0)$. 

We further assume that $X_0$ is a Zariski open set of 
a complex manifold $X$ such that 
$D=X\setminus X_0$ is a normal crossing divisor on $X$ and 
that the local monodromy of $\mathbb V_0$ around $D$ is 
unipotent. Then, by \cite[(4.12)]{schmid}, 
we can extend $(E_0, \theta_0)$ to $(E, \theta)$ on $X$, 
where 
$$
E=\Gr_F^\bullet \mathscr V =\bigoplus _p F^p/F^{p+1}
$$ 
and $$
\theta: E\to \Omega^1_X(\log D)\otimes _{\mathscr O_X}E. 
$$ 
Note that $\mathscr V$ is the canonical extension of $\mathscr V_0$ and 
$F^p$ is the canonical extension of $F^p_0$, that is,  
$$
F^p=j_*F^p_0\cap \mathscr V, 
$$ 
where $j:X_0\hookrightarrow X$ is the natural open immersion, 
for every $p$. 
\end{defn}

We need the following important calculations of curvatures 
by Griffiths. For the basic definitions and properties of 
the induced metrics and curvatures for subbundles and 
quotient bundles of a vector bundle, see \cite[\S1 and \S2]{griffiths-tu}. 

\begin{lem}\label{f-lem2.8}
We use the same notation as in Definition \ref{f-def2.7}. 
Let $F^b_0$ be the lowest piece of the Hodge filtration. 
Let $q_0$ be the metric of $F^b_0$ 
induced by the Hodge metric. 
Let $\Theta _{q_0}(F^b_0)$ be the curvature 
form of $(F^b_0, q_0)$. 
Then we have 
$$
\Theta_{q_0}(F^b_0)+(\theta^b_0)^*
\wedge \theta^b_0=0
$$ 
where $(\theta^b_0)^*$ is the adjoint of $\theta^b_0$ with 
respect to the Hodge metric {\em{(}}see, for example, \cite{griffiths-tu} and 
\cite[(7.18) Lemma]{schmid}{\em{)}}. 
Let $\mathscr L_0$ be a quotient line bundle of $F^b_0$. 
Then we have the following short exact sequence of 
locally free sheaves:  
$$
0\to \mathscr S_0\to F^b_0\to \mathscr L_0\to 0. 
$$ 
Let $A$ be 
the second fundamental form of the 
subbundle $\mathscr S_0\subset F^b_0$. 
Let $h_0$ be the induced metric of $\mathscr L_0$. 
Then we obtain 
\begin{equation*}
\begin{split}
\sqrt{-1}\Theta _{h_0} (\mathscr L_0)&=\sqrt{-1} 
\Theta_{q_0}(F^b_0)|_{\mathscr L_0} 
+\sqrt{-1} A\wedge A^*
\\& =-\sqrt{-1}(\theta^b_0)^*\wedge \theta^b_0|_{\mathscr L_0}
+\sqrt{-1}A\wedge A^*. 
\end{split}
\end{equation*}
Note that $A^*$ is the adjoint of $A$ with 
respect to $q_0$. 
Therefore, the curvature form of $(\mathscr L_0, h_0)$ is 
a semipositive smooth $(1, 1)$-form on $X_0$. 
\end{lem}

In the proof of Theorem \ref{f-thm1.1} in Section \ref{f-sec4}, 
we will investigate asymptotic behaviors of $\log h_0$, $\partial \log h_0$, 
$\partial\overline{\partial}\log h_0$ near 
the normal crossing 
divisor $D$ and see that the largest lower semicontinuous extension 
$h$ of $h_0$ on $X$ has desired properties. 

\begin{lem}\label{f-lem2.9}
We use the same notation as in Definition \ref{f-def2.7}. 
Let $q_0$ be the Hodge metric on the system of Hodge bundles 
$(E_0, \theta_0)$ induced by the original Hodge metric. 
Let $\Theta_{q_0}(E_0)$ be the curvature form of $(E_0, q_0)$. 
Then we have 
$$
\Theta_{q_0}(E_0)+\theta_0\wedge \theta^*_0+
\theta^*_0\wedge \theta_0=0
$$ 
where $\theta^*_0$ is the adjoint of 
$\theta_0$ with respect to $q_0$ {\em{(}}see, for 
example, \cite{griffiths-tu} and \cite[(7.18) Lemma]{schmid}{\em{)}}. 
Therefore, we have 
$$
\sqrt{-1}\Theta_{q_0}(E_0)=-\sqrt{-1}\theta_0\wedge \theta^*_0
-\sqrt{-1}\theta^*_0\wedge \theta_0. 
$$ 
Let $\mathscr M_0$ be a line subbundle of $E_0$ which is 
contained in the kernel of $\theta_0$ and let $h^\dag_0$ be the 
induced metric on $\mathscr M_0$. 
Then 
\begin{equation*}
\begin{split}
\sqrt{-1}\Theta_{h^\dag_0}(\mathscr M_0)&=\sqrt{-1}\Theta_{q_0}
(E_0)|_{\mathscr M_0}+\sqrt{-1}A^*\wedge A
\\&=-\sqrt{-1}\theta_0\wedge \theta^*_0|_{\mathscr M_0}
-\sqrt{-1}\theta^*_0
\wedge \theta_0|_{\mathscr M_0}+\sqrt{-1}A^*\wedge A
\\ &=-\sqrt{-1}\theta_0\wedge \theta^*_0|_{\mathscr M_0}
+\sqrt{-1}A^*\wedge A
\end{split}
\end{equation*}
where $A$ is the second fundamental form of the line subbundle 
$\mathscr M_0\subset E_0$ and $A^*$ is the adjoint of $A$ with respect to 
$q_0$. 
Therefore, the curvature of $(\mathscr M_0, h^\dag_0)$ is a seminegative 
smooth $(1, 1)$-form on $X_0$. 
\end{lem}
\end{say}

\section{Nefness}\label{f-sec3}

Let us start with the definition of nef line bundles on 
projective varieties. 

\begin{defn}[Nef line bundles]\label{f-def3.1}
A line bundle $\mathscr L$ on a projective 
variety $X$ is {\em{nef}} if $\mathscr L\cdot C\geq 0$ for 
every curve $C$ on $X$. 
\end{defn}

In this paper, we need the notion of nef locally free sheaves (or 
vector bundles) on 
projective varieties, which is a generalization of Definition \ref{f-def3.1}. 

\begin{defn}[Nef locally free sheaves]\label{f-def3.2}
A locally free sheaf (or vector bundle) $\mathscr E$ of finite rank on a 
projective variety $X$ is {\em{nef}} 
if the following equivalent conditions are satisfied: 
\begin{itemize}
\item[(i)] $\mathscr E=0$ or $\mathscr O_{\mathbb P_X(\mathscr E)}(1)$ is 
nef on $\mathbb P_X(\mathscr E)$. 
\item[(ii)] For every map from a smooth 
projective curve $f:C\to X$, every quotient line bundle 
of $f^*\mathscr E$ has nonnegative degree. 
\end{itemize}
\end{defn}

A nef locally free sheaf in Definition \ref{f-def3.2} 
was originally called 
a ({\em{numerically}}) {\em{semipositive}} sheaf in the 
literature. 

Let us recall the definition of nef line bundles in the sense of 
Demailly (see \cite[(6.11) Definition]{demailly}). 

\begin{defn}[Nef line bundles in the sense of Demailly]\label{f-def3.3}
A holomorphic line bundle $\mathscr L$ on a compact 
complex manifold $X$ is said to be {\em{nef}} if for every $\varepsilon>0$ 
there is a smooth hermitian metric $h_{\varepsilon}$ on 
$\mathscr L$ such that 
$\sqrt{-1}\Theta_{h_\varepsilon}(\mathscr L)\geq -\varepsilon \omega$, 
where $\omega$ is a fixed hermitian metric on $X$. 
\end{defn}

We can easily check: 

\begin{lem}\label{f-lem3.4}
If $X$ is projective in Definition \ref{f-def3.3}, 
then $\mathscr L$ is nef in the sense of Demailly if and 
only if $\mathscr L$ is nef in the usual sense. 
\end{lem}
\begin{proof}
It is an easy exercise. 
For the details, see \cite[(6.10) Proposition]{demailly}. 
\end{proof}

The following proposition is 
more or less well-known to the experts. We write the proof for the reader's 
convenience. 

\begin{prop}\label{f-prop3.5}
Let $X$ be a compact complex manifold and let $\mathscr L$ be 
a holomorphic line bundle equipped with 
a singular hermitian metric $h$. 
Assume that $\sqrt{-1}\Theta_h(\mathscr L)\geq 0$ and 
the Lelong number of $h$ is zero everywhere. 
Then $\mathscr L$ is a nef line bundle in the sense of 
Definition \ref{f-def3.3}.
\end{prop}

First, we give a quick proof of Proposition \ref{f-prop3.5} 
when $X$ is projective. It is an easy application of 
the Nadel vanishing theorem and the Castelnuovo--Mumford 
regularity. 

\begin{proof}[Proof of Proposition \ref{f-prop3.5} when $X$ is projective]
Let $\mathscr A$ be an ample line bundle on $X$ such that 
$|\mathscr A|$ is basepoint-free. 
By Skoda's theorem (see \cite[(5.6) Lemma]{demailly}), 
we have $\mathscr J(h^m)
=\mathscr O_X$ for 
every positive integer $m$, 
where $\mathscr J(h^m)$ is the multiplier ideal 
sheaf of $h^m$. 
Here, we used the fact that the Lelong number 
of $h$ is zero everywhere. 
By the Nadel vanishing theorem, 
$$
H^i(X, \omega_X\otimes \mathscr L^{\otimes m}\otimes 
\mathscr A^{\otimes n+1-i})=0
$$ 
for every $0<i\leq n=\dim X$ and every positive integer $m$. 
By the Castelnuovo--Mumford regularity, 
$\omega_X\otimes \mathscr L^{\otimes m}\otimes 
\mathscr A^{\otimes n+1}$ is generated by global sections for 
every positive integer $m$. 
We take a curve $C$ on $X$. 
Then $C\cdot (\omega_X\otimes \mathscr L^{\otimes m}
\otimes \mathscr A^{\otimes n+1})\geq 0$ for every 
positive integer $m$. 
This means that $C\cdot \mathscr L\geq 0$. 
Therefore, $\mathscr L$ is nef in the usual sense. 
\end{proof}

Next, we prove Proposition \ref{f-prop3.5} when $X$ is not 
necessarily projective. The proof depends on 
Demailly's approximation theorem 
for quasi-plurisubharmonic functions on 
complex manifolds (see \cite{demailly1}). 

\begin{proof}[Proof of Proposition \ref{f-prop3.5}:~general case]
Let $\omega$ be a hermitian metric on $X$ and 
let $\varepsilon$ be any positive real number. 
We fix a smooth hermitian metric $g$ on $\mathscr L$. 
Then we can write $h=ge^{-2\varphi}$, 
where $\varphi$ is an integrable function on $X$. 
Since $\sqrt{-1}\Theta_h (\mathscr L)\geq 0$, we see that 
$$\sqrt{-1}\partial \overline\partial \varphi\geq 
-\frac{1}{2}\sqrt{-1}\Theta_{g}
(\mathscr L)=:\gamma.$$ 
By Lemma \ref{f-lem2.5}, we may assume 
that $\varphi$ is quasi-plurisubharmonic. 
Note that $\gamma$ is a smooth $(1, 1)$-form on $X$. 
By \cite[Proposition 3.7]{demailly1} (see 
also \cite[(13.12) Theorem]{demailly} and \cite[Theorem 56]{demailly3}), 
we can construct a quasi-plurisubharmonic 
function $\psi_\varepsilon$ on $X$ 
with only analytic singularities (see \eqref{f-eq31} below) 
such that 
$$
\sqrt{-1}\partial \overline\partial \psi_{\varepsilon}\geq \gamma 
-\frac{1}{2}\varepsilon \omega 
$$ 
(see \cite[Proposition 3.7 (iii)]{demailly1}, \cite[(13.12) Theorem (c)]{demailly}, 
and \cite[Theorem 56 (c)]{demailly3}). 
Since the Lelong number of $h$ is zero everywhere by assumption, we obtain 
$$
0\leq \nu(\psi_{\varepsilon}, x)\leq \nu(\varphi, x)=0
$$ 
for every $x\in X$ by \cite[Proposition 3.7 (ii)]{demailly1} 
(see also \cite[(13.12) Theorem (b)]{demailly} 
and \cite[Theorem 56 (b)]{demailly3}). Therefore, the Lelong number 
of $\psi_{\varepsilon}$ is zero everywhere. 
By construction, we can easily see that $\psi_{\varepsilon}$ is smooth 
outside $\{x\in X \, |\, \psi_{\varepsilon}(x)=-\infty\}$. 
As mentioned above, 
$\psi_{\varepsilon}$ has only analytic singularities, that is, 
it can be written locally near every point $x_0\in X$ as 
\begin{equation}\label{f-eq31}
\psi_{\varepsilon}(z)=c\log \sum _{1\leq j\leq N}|g_j(z)|^2+O(1)
\end{equation}
with a family of holomorphic functions $\{g_1, \ldots, g_N\}$ defined near $x_0$ and a positive 
real number $c$ (see \cite[Definition 52]{demailly3}). 
Since $\nu(\psi_{\varepsilon}, x)=0$ for every $x\in X$, 
we obtain that $\psi_{\varepsilon}\ne -\infty$ everywhere. 
Therefore, $\psi_{\varepsilon}$ is a smooth function on $X$. 
We put $h_\varepsilon=ge^{-2\psi_{\varepsilon}}$. 
Then $h_\varepsilon$ is a smooth 
hermitian metric on $\mathscr L$ such that 
$\sqrt{-1}\Theta_{h_\varepsilon}(\mathscr L)\geq -\varepsilon \omega$. 
This means that $\mathscr L$ is a nef line bundle in the sense of 
Definition \ref{f-def3.3}. 
\end{proof}

\section{Proof of 
Theorem \ref{f-thm1.1}}\label{f-sec4}

In this section, we will prove Theorem \ref{f-thm1.1} and 
Corollary \ref{f-cor1.2}. The arguments below heavily depend on 
\cite[Section 5]{kollar}. Therefore, we strongly recommend 
the reader to see \cite[Section 5]{kollar}, especially \cite[Definition 5.3]{kollar}, 
before reading this section. 

\begin{say}\label{f-say4.1}
We put 
$
\Delta_a=\{z\in \mathbb C\, |\, |z|<a\}$, 
$\overline \Delta_a=\{z\in \mathbb C\, 
|\, |z|\leq a\}$, and $\Delta_a^*=\Delta_a
\setminus 
\{0\}$. On $\Delta^n_a$, we fix coordinates $z_1, \ldots, z_n$. 
\end{say}

Let us quickly recall the definition of {\em{nearly boundedness}} and 
{\em{almost boundedness}} due to Koll\'ar for 
the reader's convenience. 

\begin{defn}[{see \cite[Definition 5.3 (vi) and (vii)]{kollar}}]\label{f-def4.2}
On $(\Delta_a^*)^n$ with $0<a<e^{-1}$, we define 
the {\em{Poincar\'e metric}} by declaring the coframe 
$$
\left\{\frac{dz_i}{z_i \log |z_i|}, \frac{d\bar z_i}{\bar z_i \log |z_i|} \right\}
$$
to be unitary. This defines a frame of every $\Omega^k$ which 
we will refer to as the {\em{Poincar\'e frame}}. 

A function $f$ defined on a dense Zariski open set of 
$\Delta^n_a$ is called 
{\em{nearly bounded}} on $\Delta^n_a$ if $f$ is smooth 
on $(\Delta^*_a)^n$ and there are 
$C>0$, $k>0$ and 
$\varepsilon >0$ such that 
for every ordering of the coordinate functions $z_1, \ldots, z_n$ at least 
one of the following conditions is satisfied for every 
$z\in \{ z\in (\Delta_a^*)^n \, | \, |z_1| \leq \cdots 
\leq |z_n|\}$. 
\begin{itemize}
\item[$(\mathbf{a})$:] $|f|\leq C$, 
\item[$(\mathbf{b})$:] 
$|z_1|\leq \exp (-|z_m|^{-\varepsilon})$ and 
$|f| \leq C(-\log |z_m|)^k$ for some $2\leq m\leq n$. 
\end{itemize}

A form $\eta$ defined on a dense Zariski open set 
of $\Delta^n_a$ is called 
{\em{nearly bounded}} on $\Delta^n_a$ if the 
coefficient functions are nearly bounded 
on $\Delta^n_a$ when 
we write $\eta$ in terms of the Poincar\'e frame. 
If $\eta_1$ and $\eta_2$ are nearly 
bounded on the same $\Delta^n_a$, 
then $\eta_1\wedge \eta_2$ is nearly bounded on $\Delta^n_a$. 

A form $\eta$ defined on a dense Zariski open set of 
$\Delta^n_a$ is called {\em{almost bounded}} on $\Delta^n_a$ if 
there is a proper bimeromorphic map $p:W\to \Delta_a^n$ such 
that $W$ is smooth and every $w\in W$ has a neighborhood 
where $p^*\eta$ is nearly bounded. 
\end{defn}

\begin{rem}\label{f-rem4.3}
The definition of nearly boundedness and almost boundedness in Definition 
\ref{f-def4.2} is slightly different from Koll\'ar's 
original one (see \cite[Definition 5.3 (vii)]{kollar}). 
We think that it is a kind of clarification. 
Of course, everything in \cite[Section 5]{kollar} works well for our 
definition. 
\end{rem}

\begin{say}[Proof of Theorem \ref{f-thm1.1}]\label{f-say4.4}
We fix a smooth hermitian metric $g$ on $\mathscr L$. 
The Hodge metric induces a smooth hermitian metric 
$h_0$ on $\mathscr L|_{X_0}$. 
Then we can 
write 
$$
h_0=ge^{-2\varphi_0}
$$
for some smooth function $\varphi_0$ on $X_0$. 
We use the same notation as in Lemma \ref{f-lem2.8}. 
Let $\mathscr V$ be the 
canonical extension of $\mathscr V_0=\mathbb V_0\otimes \mathscr O_{X_0}$. 
Let $q_0$ be the Hodge metric on $\mathscr V_0$. 
For simplicity, we use the 
same notation $q_0$ to denote $(q_0)|_{F^b_0}$, 
that is, the metric on $F^b_0$ induced by the metric 
$q_0$ on $\mathscr V_0$. 
Let $P$ be an arbitrary point of $X$. 
We take a suitable local coordinate $(z_1, \ldots, z_n)$ centered at $P$ 
and a small positive real number $a$ with $a<e^{-1}$. 
Then, by \cite[Theorem 5.21]{cks} (see also 
\cite{kashiwara} and \cite[Claim 7.8]{viehweg-zuo}), we can write 
$$
\mathscr V|_{\Delta^n_a}\simeq \bigoplus _{i=1}^r\mathscr O_{\Delta^n_a}e_i(z), 
$$ 
where $e_i(z)\in \Gamma (\Delta^n_a, \mathscr V)$, such that 
\begin{equation}\label{f-eq41}
q_0(e_i(z), e_i(z))\leq C_1 (-\log|z_1|)^{a_1}\cdots (-\log |z_n|)^{a_n}
\end{equation} 
for $z\in (\Delta^*_a)^n$, where $a_1, \ldots, a_n$ 
are some positive integers and 
$C_1$ is a large positive real number. 
By making $a$ smaller, we may further assume that 
$$
\mathscr L|_{\Delta^n_a}\simeq \mathscr O_{\Delta^n_a}e(z), 
$$ 
where $e(z)\in \Gamma (\Delta^n_a, \mathscr L)$ is a nowhere vanishing section of 
$\mathscr L$ on $\Delta^n_a$. 
We take a lift $f(z)\in \Gamma (\Delta^n_a, F^b)$ of $e(z)$, that is, 
$p(f(z))=e(z)$, where $p:F^b\to \mathscr L$. 
Then we can write 
\begin{equation}\label{f-eq42} 
f(z)=f_1(z)e_1(z)+\cdots +f_r(z)e_r(z), 
\end{equation} 
where $f_i(z)$ is a holomorphic function on $\Delta^n_a$ for every 
$i$. 
By making $a$ smaller again, we may assume that 
$f_i(z)$ is holomorphic in a neighborhood of $(\overline \Delta_a)^n$. 
Of course, we may further assume that 
$e(z)\ne 0$ in a neighborhood of $(\overline \Delta_a)^n$. 
By \eqref{f-eq41} and \eqref{f-eq42}, 
we obtain that there exists some large positive 
real number $C_2$ such that 
$$
q_0(f(z), f(z))\leq C_2(-\log|z_1|)^{a_1}\cdots (-\log |z_n|)^{a_n}
$$ holds 
for $z\in (\Delta^*_a)^n$. 
Therefore, 
\begin{equation*}
\begin{split}
C_3e^{-2\varphi_0(z)}&\leq g(e(z), e(z))e^{-2\varphi_0(z)}\\ 
&=h_0(e(z), e(z))\\
&\leq q_0(f(z), f(z))\leq C_2(-\log|z_1|)^{a_1}\cdots (-\log |z_n|)^{a_n}
\end{split}
\end{equation*}
for $z\in (\Delta^*_a)^n$, where $$C_3=\min _{z\in (\overline \Delta_a)^n}
g(e(z), e(z))>0.$$ 
Thus, 
$$
-\varphi_0(z)\leq \log \left( C\left(-\log |z_1|\right)^{a_1}
\cdots \left(-\log |z_n|\right)^{a_n}\right)
$$ 
holds for $z\in (\Delta^*_a)^n$, where 
$C$ is some large positive real number. 
By applying similar arguments to the dual line bundle $\mathscr L^{\vee}$, 
we may further assume that 
$$
\varphi_0(z)\leq \log \left( C\left(-\log |z_1|\right)^{a_1} 
\cdots \left(-\log |z_n|\right)^{a_n}\right)
$$ 
holds for $z\in (\Delta^*_a)^n$. 
Let $\varphi$ be the smallest upper 
semicontinuous function that extends $\varphi_0$ 
to $X$. More explicitly, 
$$
\varphi(z)=\lim _{\varepsilon \to 0}\sup _{w\in \Delta^n_\varepsilon\cap X_0}\varphi_0(w), 
$$
where $\Delta^n_\varepsilon$ is a polydisc on $X$ centered at $z\in X$. 
Then, by Lemma \ref{f-lem4.6}, we obtain: 
\begin{lem}\label{f-lem4.5}
$\varphi$ is locally integrable on $X$. 
\end{lem}
\begin{proof}[Proof of Lemma \ref{f-lem4.5}] 
Let $P$ be an arbitrary point of $X$. 
In a small open neighborhood of $P$, we have 
$$
0\leq \varphi_{\pm} (z) \leq \log \left( C\left(-\log |z_1|\right)^{a_1} 
\cdots \left(-\log |z_n|\right)^{a_n}\right)
$$where 
$\varphi_+=\max \{ \varphi, 0\}$ and $\varphi_-=\varphi_+-\varphi$. 
By Lemma \ref{f-lem4.6} below, we obtain that $\varphi$ is 
locally integrable on $X$. 
\end{proof}

We have already used: 

\begin{lem}\label{f-lem4.6}
We have 
$$
\int _0^a r\log (-\log r)dr<\infty
$$ 
for $0<a <e^{-1}$. 
\end{lem}
\begin{proof}[Proof of Lemma \ref{f-lem4.6}]
We put $t=-\log r$. 
Then we can easily check 
\begin{equation*}
\int _0^a r\log (-\log r)dr =\int _{-\log a} ^\infty 
e^{-2t} (\log t) dt\leq \int ^{\infty}_{-\log a} t e^{-2t} dt 
\leq \int^\infty_{-\log a}e^{-t}dt=a<\infty 
\end{equation*} 
by direct calculations. 
\end{proof}

We put 
$$
h=ge^{-2\varphi}. 
$$ 
Then $h$ is a singular hermitian metric on $\mathscr L$ 
in the sense of Definition \ref{f-def2.2}. 
The following lemma is essentially contained 
in \cite[Propositions 5.7 and 5.15]{kollar}. 

\begin{lem}\label{f-lem4.7}
Let $P$ be an arbitrary point of $X$. 
Then $\partial\varphi_0$ and $\bar\partial \partial \varphi_0$ are 
almost bounded in a neighborhood 
of $P\in X$. 
More precisely, 
there exists $\Delta_a^n$ on $X$ centered at $P$ for some $0<a<e^{-1}$ 
such that $\varphi_0$, $\partial \varphi_0$, and 
$\bar\partial \partial \varphi_0$ are smooth 
on $(\Delta_a^*)^n$ and that 
$\partial \varphi_0$ and 
$\bar\partial \partial \varphi_0$ are 
almost bounded on $\Delta^n_a$. 
\end{lem}

\begin{proof}[Proof of Lemma \ref{f-lem4.7}]
We consider the following short exact sequence: 
$$
0\to \mathscr S\to F^b\to \mathscr L\to 0. 
$$
We fix smooth hermitian metrics $g_1, g_2$ and $g$ on $\mathscr S, F^b$, 
and $\mathscr L$, respectively. 
We assume that $g_1=g_2|_{\mathscr S}$ and 
that 
$g$ is the orthogonal complement of $g_1$ in $g_2$. 
Let $h_1$ and $h_2$ be the induced 
Hodge metrics on $\mathscr S_0=\mathscr S|_{X_0}$ 
and 
$F^b_0$, respectively. 
By applying the calculations in \cite[Section 5]{kollar} 
to $\det \mathscr S$ and $\det F^b$, we obtain 
$\det h_1=\det g_1\cdot e^{-\varphi_1}$ and 
$\det h_2=\det g_2\cdot e^{-\varphi_2}$ on $X_0$ such that 
$\partial \varphi_1, \bar\partial \partial \varphi_1, 
\partial \varphi_2$, and $\bar\partial \partial \varphi_2$ are 
almost bounded in a neighborhood of $P$. 
More precisely, we can take a polydisc $\Delta^n_a$ centered at 
$P$ for some $0<a<e^{-1}$ and a composite of permissible 
blow-ups $p:W\to \Delta^n_a$ (see 
\cite[5.9]{kollar} and \cite[Theorem 3.5.1]{wlo}) such that 
$\varphi_1$ and $\varphi_2$ are smooth 
on $(\Delta^*_a)^n$ and that 
every $w\in W$ has a neighborhood $\Delta^n_{a'_w}$ centered at 
$w\in W$ for some 
$0<a'_w<e^{-1}$ where $p^*(\partial \varphi_1)$, $p^*(\overline \partial 
\partial \varphi_1$), $p^*(\partial \varphi_2)$, and 
$p^*(\overline \partial \partial \varphi_2)$ are nearly bounded on 
$\Delta^n_{a'_w}$.  
For the details, see \cite[Propositions 5.7 and 5.15]{kollar}. 
By construction, $\varphi_0=-\varphi_1+\varphi_2$. 
Therefore, $\varphi_0$ is smooth on $(\Delta^*_a)^n$, and 
$p^*(\partial \varphi_0)$ and $p^*(\overline \partial \partial \varphi_0)$ 
are nearly bounded on $\Delta^n_{a'_w}$. 
This means that $\varphi_0$, $\partial \varphi_0$, 
and $\overline \partial \partial \varphi_0$ are smooth 
on $(\Delta^*_a)^n$ and that 
$\partial \varphi_0$ and $\overline \partial \partial \varphi_0$ are almost 
bounded on $\Delta^n_a$.    
\end{proof}

We prepare an easy lemma. 

\begin{lem}\label{f-lem4.8}
We assume $0<a<e^{-1}$. 
We have 
$$
\int _0^a \frac{\log (-\log r)}{-\log r} dr <\infty. 
$$
\end{lem}

\begin{proof}[Proof of Lemma \ref{f-lem4.8}]
We put $t=-\log r$. 
Then $r=e^{-t}$. 
We have 
\begin{equation*}
\begin{split}
\int ^a_0 \frac{\log (-\log r)}{-\log r} dr &
=\int^{-\log a}_{\infty} \frac{\log t}{t}(-e^{-t})dt \\ 
&= \int ^{\infty}_{-\log a} \frac{\log t}{t} e^{-t}dt \\ 
&\leq \int ^{\infty}_{-\log a} e^{-t} dt =a<\infty. 
\end{split}
\end{equation*}
This is what we wanted. 
\end{proof}

The following lemma is missing in \cite[Section 5]{kollar}. 
This is because it is sufficient to 
consider the asymptotic behaviors of $\partial \varphi_0$ and 
$\bar\partial \partial \varphi_0$ for the purpose of \cite[Section 5]{kollar}. 

\begin{lem}\label{f-lem4.9}
Let $\eta$ be a smooth $(2n-1)$-form on $\Delta_a^n$ with 
compact support. 
We put $$
S_{\vec{\varepsilon}} =\{ z\in \Delta_a^n \, |\, 
|z_i|\geq \varepsilon ^i \text{\ for every $i$ and\ }  
|z_{i_0}|=\varepsilon^{i_0} \text{\ for some $i_0$}\} 
$$ 
where $\vec{\varepsilon}=(\varepsilon^1, \ldots, \varepsilon^n)$ 
with $\varepsilon ^i>0$ for every $i$. 
Then there is a sequence $\{\vec{\varepsilon}_k\}$ with $
\vec{\varepsilon}_k\searrow 0$ 
such that  
$$
\lim _{k\to \infty} \int _{S_{\vec{\varepsilon}_k}} \varphi \eta=0. 
$$
\end{lem} 
\begin{proof}[Proof of Lemma \ref{f-lem4.9}]
We put 
$$
S_{\varepsilon, 1}=\{z\in \Delta_a^n \, | \, |z_1|=\varepsilon\}. 
$$ 
Then it is sufficient to prove that 
$$
\lim _{k\to \infty} \int _{S_{\varepsilon_k, 1}}\varphi\eta=0 
$$ for some sequence $\{\varepsilon _k\}$ with $\varepsilon _k\searrow 
0$. 
Without loss of generality, we may assume that $\eta$ is a real 
$(2n-1)$-form by considering $\frac{\eta+\bar\eta}{2}$ and 
$\frac{\eta-\bar\eta}{2\sqrt{-1}}$. 
Let us consider the real $1$-form 
$$
\omega =\frac{1}{\left(2 (-\log |z_1|)^2\right)^{1/2}}\left( 
\frac{dz_1}{z_1}+\frac{d\bar z_1}{\bar z_1}\right). 
$$
This form is orthogonal to the foliation 
$S_{\varepsilon, 1}$ 
and has length one everywhere by the Poincar\'e metric. 
We consider the vector field 
$$
v=\frac{1}{\left(2 (-\log |z_1|)^2\right)^{1/2}}\left( z_1 (\log |z_1|)^2
\frac{\partial}{\partial z_1}+\bar z_1(\log |z_1|)^2\frac{\partial}{\partial 
\bar z_1}\right),  
$$ which is dual to $\omega$. 
We fix $\varepsilon$ with $0<\varepsilon <a<e^{-1}$.  
We consider the flow $f_t$ on $\Delta_a^*\times \Delta_a^{n-1}$ 
corresponding to $-v$. 
We can explicitly write 
$$
f_t: [0, \infty)\times S_{\varepsilon, 1}\to 
\Delta_a^*\times \Delta_a^{n-1}
$$ 
by 
\begin{equation}\label{f-eq43}
(t, (w, z_2, \cdots, z_n))\mapsto 
\left( \frac{w}{\varepsilon}\exp \left(-\exp  \left( 
\frac{1}{\sqrt{2}}t+\log (-\log \varepsilon)\right)\right), 
z_2, \cdots, z_n\right). 
\end{equation} 
Therefore, by using the flow $f_t$, 
we can parametrize $\{z\in \mathbb C \,|\, 0< |z|\leq \varepsilon\} \times 
\Delta_a^{n-1}$ by $[0, \infty)\times S_{\varepsilon, 1}$. 
If we write 
$$\omega\wedge \varphi\eta=f(z)dV,$$ 
where $dV$ is the standard 
volume form of $\mathbb C^n$, then 
we put $$(\omega\wedge \varphi\eta)^+=\max \{ f(z), 0\}dV$$ and 
$$(\omega\wedge \varphi\eta)^-=
(\omega\wedge \varphi\eta)^+-\omega\wedge \varphi\eta. $$ 
We can easily see that 
$$
\int _{\Delta_a^n}(\omega\wedge \varphi\eta)^\pm<\infty  
$$ by Lemmas \ref{f-lem4.6} and \ref{f-lem4.8}. 
Therefore, we obtain 
\begin{equation}\label{f-eq44}
\int_{[0, \infty)\times S_{\varepsilon, 1}}(\omega\wedge 
\varphi\eta)^\pm<\infty. 
\end{equation}
The image of $\{t\}\times S_{\varepsilon, 1}$ in $\Delta_a^n$ is 
$S_{\varepsilon(t), 1}$ with $0<\varepsilon(t)\leq \varepsilon$. 
By \eqref{f-eq43}, we have 
$$ 
\varepsilon(t)=\exp \left(-\exp\left(\frac{1}{\sqrt{2}}t+\log (-\log \varepsilon)
\right)\right). 
$$ 
We note that 
$\omega$ is orthogonal to $S_{\varepsilon(t), 1}$ and 
unitary. 
More explicitly, we can directly check 
$$
f_t^*\omega=-dt. 
$$
Therefore, the above integral \eqref{f-eq44} 
transforms to 
$$
\int_{[0, \infty)}\left(\int_{S_{\varepsilon(t), 1}}(\varphi\eta)^\pm\right) dt<\infty. 
$$ 
Note that $(\varphi\eta)^\pm$ is defined by 
$$
f_t^*(\omega\wedge \varphi\eta)^\pm=-dt \wedge 
(\varphi\eta)^\pm. 
$$
This can happen only if 
$$
\int _{S_{\varepsilon(t_k), 1}}(\varphi \eta)^\pm\to 0
$$ 
for some sequence $\{t_k\}$ with $t_k\nearrow \infty$. 
This implies that we can take a sequence $\{\varepsilon_k\}$ with 
$\varepsilon _k \searrow 0$ such that 
$$
\lim_{k\to \infty}\int _{S_{\varepsilon_k, 1}}\varphi\eta=0. 
$$ 
Therefore, we have a desired sequence $\{\vec{\varepsilon}_k\}$. 
\end{proof}

\begin{rem}\label{f-rem4.10}
The real $1$-form $\omega$ and the corresponding flow $f_t$ in the 
proof of Lemma \ref{f-lem4.9} are different from the 
$1$-form $\omega$ and 
the flow $v_t$ in the proof of \cite[Proposition 5.16]{kollar}, respectively. 
\end{rem}

By combining the proof of \cite[Proposition 5.16]{kollar} and 
the proof of Lemma \ref{f-lem4.9}, we have: 

\begin{lem}\label{f-lem4.11}
Let $\eta$ be a nearly bounded $(2n-1)$-form on $\Delta_a^n$ with 
compact support. 
Then there exists a sequence 
$\{\vec{\varepsilon'}_k\}$ with $\vec{\varepsilon'}_k\searrow 0$ such that 
$$
\lim_{\vec{\varepsilon'}_k\searrow 0}\int _{S_{\vec{\varepsilon'}_k}}
\eta=0. 
$$
\end{lem}

We leave the details of Lemma \ref{f-lem4.11} to 
the reader (see the proof of \cite[Proposition 5.16]{kollar} and 
the proof of Lemma \ref{f-lem4.9}). 

By Lemmas \ref{f-lem4.7} and \ref{f-lem4.9}, we have the following 
lemma. 

\begin{lem}\label{f-lem4.12}
Let $\eta$ be a smooth 
$(2n-2)$-form on $\Delta_a^n$ with compact support. 
We further assume that $\partial \varphi_0$ and $\bar\partial \partial 
\varphi_0$ are nearly bounded on $\Delta^n_a$. 
Then 
$$
\int _{\Delta_a^n} \varphi\partial \bar\partial \eta=
\int _{\Delta_a^n}\partial \bar\partial \varphi_0\wedge \eta. 
$$
Note that the right hand side is an improper integral. 
Therefore, we obtain 
$$
\int _{\Delta_a^n} \partial \bar\partial\varphi\wedge \eta=
\int _{\Delta_a^n}\partial \bar\partial \varphi_0\wedge \eta,  
$$ 
where we take $\partial \bar\partial$ of $\varphi$ as a current. 
\end{lem}

\begin{proof}[Proof of Lemma \ref{f-lem4.12}]
We put 
$$
V_{\vec{\varepsilon}_k}=\{z\in \Delta_a^n \, |\, 
|z_i|\geq \varepsilon_k^i {\text{\ for 
every $i$}}\}
$$ where 
$\vec{\varepsilon}_k=(\varepsilon _k^1, \cdots, 
\varepsilon _k^n)$ with $\varepsilon _k^i>0$ for every $i$. 
Then 
\begin{equation*}
\begin{split}
\int _{\Delta_a^n} \varphi\partial \bar\partial \eta  &= 
\lim_{\vec{\varepsilon}_k\searrow 0} \int _{V_{\vec{\varepsilon}_k}} \varphi_0 
\partial \bar\partial \eta \\ 
&=\lim_{\vec{\varepsilon}_k\searrow 0} \int _{V_{\vec{\varepsilon}_k}} 
d(\varphi_0\bar\partial \eta)
-\lim_{\vec{\varepsilon'}_k\searrow 0} \int _{V_{\vec{\varepsilon'}_k}} 
\partial \varphi_0\wedge \bar\partial \eta
\\ 
& =\lim_{\vec{\varepsilon}_k\searrow 0} \int _{S_{\vec{\varepsilon}_k}} 
\varphi_0\bar\partial \eta+
\lim_{\vec{\varepsilon'}_k\searrow 0} \int _{V_{\vec{\varepsilon'}_k}} 
d(\partial \varphi_0\wedge \eta)
-\lim_{\vec{\varepsilon'}_k\searrow 0} \int _{V_{\vec{\varepsilon'}_k}} 
\bar\partial \partial \varphi_0 \wedge \eta
\\ & = \lim_{\vec{\varepsilon'}_k\searrow 0} \int _{V_{\vec{\varepsilon'}_k}} 
\partial \bar\partial \varphi_0 \wedge \eta
\\ 
&= \int _{\Delta_a^n} \partial \bar\partial \varphi_0 \wedge \eta. 
\end{split}
\end{equation*}
The first equality holds since $\varphi$ is locally integrable. 
The second one follows from integration by parts. 
Note that $\varphi_0$ is smooth in a neighborhood of 
$V_{\vec{\varepsilon}_k}$. 
We also note that 
$$
\lim_{\vec{\varepsilon}_k\searrow 0} \int _{V_{\vec{\varepsilon}_k}} 
\partial \varphi_0\wedge \bar\partial \eta=
\lim_{\vec{\varepsilon'}_k\searrow 0} \int _{V_{\vec{\varepsilon'}_k}} 
\partial \varphi_0\wedge \bar\partial \eta
$$ 
holds. 
The third one follows from Stokes' theorem and integration by parts. 
We obtain the fourth one by Lemmas \ref{f-lem4.9} and 
\ref{f-lem4.11}. 
Note that 
$$
\int _{V_{\vec{\varepsilon'}_k}} 
d(\partial \varphi_0\wedge \eta)=\int _{S_{\vec{\varepsilon'}_k}} 
\partial \varphi_0\wedge \eta
$$ 
by Stokes' theorem. 
The final one follows from \cite[Proposition 5.16 (i)]{kollar}. 
\end{proof}

\begin{lem}\label{f-lem4.13}
Let $\eta$ be a smooth 
$(2n-2)$-form on $\Delta_a^n$ with compact support. 
We assume that $\partial \varphi_0$ and $\bar\partial \partial \varphi_0$ are 
almost bounded on $\Delta^n_a$. 
Then 
$$
\int _{\Delta_a^n} \varphi\partial \bar\partial \eta=
\int _{\Delta_a^n}\partial \bar\partial \varphi_0\wedge \eta. 
$$
\end{lem}
\begin{proof}[Proof of Lemma \ref{f-lem4.13}]
By assumption, $\partial \varphi_0$ and 
$\bar\partial \partial \varphi_0$ are almost bounded on $\Delta^n_a$. 
Therefore, after taking some suitable blow-ups and 
a suitable partition of unity, we can apply Lemma \ref{f-lem4.12}. 
Then we obtain the desired equality. 
\end{proof}

\begin{lem}\label{f-lem4.14}
Let $\eta_1$ and $\eta_2$ be a smooth 
$(2n-2)$-form and a smooth 
$(2n-3)$-form on $X$ with 
compact support, respectively. 
Then 
\begin{equation}\label{f-eq45}
\int_X\sqrt{-1} \Theta_{h_0}(\mathscr L|_{X_0})\wedge \eta_1<\infty
\end{equation}
and 
\begin{equation}\label{f-eq46}
\int _X\sqrt{-1}\Theta_{h_0}(\mathscr L|_{X_0})\wedge d\eta_2=0. 
\end{equation}
Therefore, $\sqrt{-1}\Theta_{h_0}(\mathscr L|_{X_0})$ can be 
extended to a closed positive current $T$ on $X$ by improper 
integrals. 
We note that $\sqrt{-1}\Theta_{h_0}(\mathscr L|_{X_0})$ is a 
semipositive smooth $(1, 1)$-form on $X_0$ {\em{(}}see 
Lemma \ref{f-lem2.8}{\em{)}}. 
\end{lem}
\begin{proof}[Proof of Lemma \ref{f-lem4.14}] 
We note that 
$$
\sqrt{-1}\Theta_{h_0}(\mathscr L|_{X_0})=\sqrt{-1}\Theta_g(\mathscr L)|_{X_0} 
+2\sqrt{-1}\partial \overline\partial \varphi_0
$$ 
by definition and that $\sqrt{-1}\Theta_g(\mathscr L)$ is a $d$-closed smooth 
$(1, 1)$-form on $X$. 
Therefore, it is sufficient to prove that 
\begin{equation}\label{f-eq47}
\int _{\Delta^n_a}\sqrt{-1}\partial \overline \partial \varphi_0\wedge \eta_1<\infty
\end{equation} 
and 
\begin{equation}\label{f-eq48}
\int_{\Delta^n_a}\partial \overline\partial \varphi_0\wedge d\eta_2=0
\end{equation}
by taking some suitable partition of unity. 
We see that \eqref{f-eq47} and \eqref{f-eq48} follow from 
\cite[Corollary 5.17]{kollar} 
since $\partial \overline \partial \varphi_0$ is almost 
bounded on $\Delta^n_a$ (see Lemma \ref{f-lem4.7}). 
More precisely, by taking some suitable blow-ups and a suitable partition of 
unity, we can reduce the problems to the case where 
$\partial \overline\partial \varphi_0$ is nearly 
bounded on some polydisc $\Delta^n_a$. 
Then \eqref{f-eq47} follows from \cite[Proposition 5.16 (i)]{kollar}. 
By \cite[Proposition 5.16 (i)]{kollar}, 
integration by parts, Stokes' theorem, and Lemma \ref{f-lem4.11}, 
we can directly check that 
$$
\int _{\Delta^n_a}\partial \overline\partial \varphi_0\wedge d\eta_2=0
$$ 
as in the proof of Lemma \ref{f-lem4.12}.
\end{proof}
By Lemma \ref{f-lem4.13}, we can see that 
\begin{equation}\label{f-eq49}
\sqrt{-1}\Theta_h(\mathscr L)=\sqrt{-1} \Theta_g(\mathscr L) +
2\sqrt{-1}\partial 
\overline \partial \varphi
\end{equation}
coincides with $T$. 
Note that we took $\partial \overline\partial$ of $\varphi$ as a 
current in \eqref{f-eq49}. 
In particular, 
$$
\sqrt{-1}\Theta_h(\mathscr L)\geq 0, 
$$ 
that is, $\sqrt{-1}\Theta_h(\mathscr L)$ is a closed 
positive current on $X$. 
By Lemma \ref{f-lem2.5}, $\varphi$ is a quasi-plurisubharmonic 
function since $\varphi$ is the smallest upper semicontinuous 
function that extends $\varphi_0$ to $X$. 

Finally, we prove: 

\begin{lem}\label{f-lem4.15}
Let $\varphi$ be a 
quasi-plurisubharmonic function on $\Delta_a^n$ for some 
$0<a<e^{-1}$. 
Assume that there exist some positive integers $a_1, \cdots, a_n$ and 
a positive real number $C$ such that 
$$
-\varphi(z)\leq \log \left( C\left(-\log |z_1|\right)^{a_1} 
\cdots \left(-\log |z_n|\right)^{a_n}\right)
$$ 
holds for all $z\in (\Delta_a^*)^n$. 
Then the Lelong number of $\varphi$ at $0$ is zero. 
\end{lem}

\begin{proof}
We denote the Lelong number of $\varphi$ at $x$ by $\nu(\varphi, x)$. 
We can easily see that 
$$
0\leq \nu(\varphi, 0)=\liminf_{z\to 0} \frac{\varphi (z)}{\log |z|}
\leq \liminf _{z\to 0} \frac{\log \left( C\left(-\log |z_1|\right)^{a_1} 
\cdots \left(-\log |z_n|\right)^{a_n}\right)}{-\log |z|}\leq 0
$$ 
holds. Therefore, the Lelong number $\nu(\varphi, 0)$ 
of $\varphi$ at $0$ is zero. 
\end{proof}

Thus we obtain Theorem \ref{f-thm1.1} by 
Lemma \ref{f-lem4.15}. 
\end{say}

Now Corollary \ref{f-cor1.2} is almost obvious by Theorem \ref{f-thm1.1}. 

\begin{proof}[Proof of Corollary \ref{f-cor1.2}]
We put $\pi:Y=\mathbb P_X(F^b)\to X$ and $Y_0=\pi^{-1}(X_0)$. 
We consider the variation of Hodge structure $\pi^*V_0$ on $Y_0$. 
Then $\pi^*F^b$ is the canonical extension of the lowest piece of the Hodge 
filtration. 
By applying Theorem \ref{f-thm1.1} to 
the natural map $\pi^*F^b\to \mathscr O_{\mathbb P_X(F^b)}(1)
\to 0$, 
we obtain a singular hermitian metric on $\mathscr O_{\mathbb P_X(F^b)}
(1)$ with 
the desired properties. 
\end{proof}

\section{Proof of 
Theorem \ref{f-thm1.5}}\label{f-sec5}

In this section, we will prove Theorem \ref{f-thm1.5} and 
Corollary \ref{f-cor1.6}. We will only explain how to modify 
the arguments in Section \ref{f-sec4}. 

\begin{say}[Proof of Theorem \ref{f-thm1.5}]\label{f-say5.1} 
Let $\{F^p_0\}$ be the Hodge filtration of the polarizable 
variation of $\mathbb R$-Hodge structure $V_0=(\mathbb V_0, 
F_0)$ on $X_0$. 
We put 
$$
0=F^{b+1}_0\subsetneq F^b_0\subseteq \cdots 
\subseteq F^{a+1}_0\subsetneq F^a_0=\mathscr V_0:=\mathbb V_0\otimes 
\mathscr O_{X_0}. 
$$ 
By assumption, $\mathscr M$ is a holomorphic 
line subbundle of $\bigoplus _p\Gr^p_F \mathscr V$. 
Therefore, $\mathscr M$ is naturally a holomorphic 
line subbundle of $\mathscr Q:=\bigoplus _{p=a+1}^{b+1} \mathscr V/F^p$. 
Then we have the following big commutative diagram of 
holomorphic vector bundles on $X$. 
$$
\xymatrix{&&&0\ar[d]&\\
&0\ar[d]&&\mathscr M\ar[d]&\\
0 \ar[r]& \mathscr S'\ar[d]
\ar[r]& \underset{\text{finite}}{\bigoplus} 
\mathscr V\ar[r]\ar@{=}[d]& \mathscr Q\ar[r]\ar[d]&0\\
0 \ar[r]& \mathscr S\ar[d]
\ar[r]& \underset{\text{finite}}{\bigoplus} 
\mathscr V\ar[r]& \mathscr Q'\ar[r]\ar[d]&0\\
&\mathscr M\ar[d]&&0&\\
&0&&&
}
$$ 
We note that $\mathscr S'=\bigoplus _{p=a+1}^{b+1}F^p$ and 
$\mathscr Q'=\mathscr Q/\mathscr M$ and 
that $\mathscr S$ is the kernel of the naturally induced surjection 
$\bigoplus _{\text{finite}}\mathscr V\to \mathscr Q'$. 
By taking the dual of the above 
commutative diagram, $\mathscr M^\vee$ is a quotient 
bundle of $\mathscr Q^\vee$ and 
$\mathscr Q^\vee$ is a subbundle of $\underset{\text{finite}}
{\bigoplus}\mathscr V^\vee$. 
Therefore, we can apply the same arguments as in Section \ref{f-sec4} to 
$\mathscr M^\vee$ by considering the polarizable variation of 
$\mathbb R$-Hodge structure $\underset{\text{finite}}{\bigoplus}
V^\vee_0$. 
Then we see that the Hodge metric of $\underset{\text{finite}}{\bigoplus} 
V^\vee_0$ induces the desired singular hermitian metric 
$h$ on $\mathscr M^\vee$ by Lemma \ref{f-lem2.9}. 
\end{say}

Finally, we give a proof of Corollary \ref{f-cor1.6}. 

\begin{proof}[Proof of Corollary \ref{f-cor1.6}]
We put $\pi:Y=\mathbb P_X(A^\vee)\to X$ and 
$Y_0=\pi^{-1}(X_0)$. 
We consider $\pi^*V_0$ on $Y_0$. 
Then $\pi^*A$ is contained in the 
kernel of the Higgs field 
$$
\pi^*\theta: \Gr_F^\bullet\pi^*\mathscr V\to 
\Omega^1_Y(\log \pi^*D)\otimes 
\Gr_F^\bullet\pi^*\mathscr V. 
$$
By applying Theorem \ref{f-thm1.5} to 
the line subbundle $\mathcal O_{\mathbb P_X(A^\vee)}(-1)$ 
of $\pi^*A$, we obtain the desired result. 
\end{proof}


\begin{thebibliography}{Kaw3}

\bibitem[B1]{brunebarbe} 
Y.~Brunebarbe, Symmetric differentials and variations of 
Hodge structures, to appear in J. Reine Angew. Math.

\bibitem[B2]{brunebarbe2} 
Y.~Brunebarbe, 
Semi-positivity from Higgs bundles, preprint (2017). 

\bibitem[CKS]{cks} 
E.~Cattani, A.~Kaplan, W.~Schmid, 
Degeneration of Hodge structures, 
Ann. of Math. (2) {\textbf{123}} (1986), no. 3, 457--535. 

\bibitem[D1]{demailly1} 
J.-P.~Demailly, Regularization of closed positive currents and intersection 
theory, J. Algebraic Geom. \textbf{1} (1992), no. 3, 361--409.

\bibitem[D2]{demailly} 
J.-P.~Demailly, {\em{Analytic methods in algebraic geometry}}, 
Surveys of Modern Mathematics, \textbf{1}. International 
Press, Somerville, MA; Higher Education Press, Beijing, 2012.

\bibitem[D3]{demailly3} 
J.-P.~Demailly, 
Applications of pluripotential theory to algebraic geometry, 
{\em{Pluripotential theory}, 143--263, Lecture 
Notes in Math., \textbf{2075}, Fond. CIME/CIME Found. Subser., Springer, 
Heidelberg, 2013.
}
\bibitem[FF]{fujino-fujisawa} 
O.~Fujino, T.~Fujisawa, 
Variations of mixed Hodge structure and semipositivity theorems, 
Publ. Res. Inst. Math. Sci. \textbf{50} (2014), no. 4, 589--661.

\bibitem[FFS]{fujino-fujisawa-saito} 
O.~Fujino, T.~Fujisawa, M.~Saito, 
Some remarks on the semipositivity theorems, 
Publ. Res. Inst. Math. Sci. \textbf{50} (2014), no. 1, 85--112.

\bibitem[Fuj]{fujisawa} 
T.~Fujisawa, A remark on semipositivity theorems, preprint (2017). 

\bibitem[GT]{griffiths-tu} 
P.~Griffiths, L.~Tu, 
Curvature properties of the Hodge bundles, 
{\em{Topics in transcendental algebraic 
geometry (Princeton, N.J., 1981/1982)}}, 29--49, 
Ann. of Math. Stud., \textbf{106}, Princeton Univ. Press, Princeton, NJ, 1984.

\bibitem[HPS]{hps} 
C.~Hacon, M.~Popa, C.~Schnell, 
Algebraic fiber spaces over abelian varieties:~around a 
recent theorem by Cao and P\u aun, 
to appear in the 
Contemporary Mathematics volume in honor of L. Ein's 60th birthday. 

\bibitem[Kas]{kashiwara} 
M.~Kashiwara, 
The asymptotic behavior of a variation of 
polarized Hodge structure, 
Publ. Res. Inst. Math. Sci. \textbf{21} (1985), no. 4, 853--875.

\bibitem[Kaw1]{kawamata1}
Y.~Kawamata, 
Characterization of abelian varieties, 
Compositio Math. \textbf{43} (1981), no. 2, 253--276.

\bibitem[Kaw2]{kawamata2} 
Y.~Kawamata, Kodaira dimension of certain algebraic fiber spaces, 
J. Fac. Sci. Univ. Tokyo Sect. IA Math. \textbf{30} (1983), no. 1, 1--24.

\bibitem[Kaw3]{kawamata3}
Y.~Kawamata, 
On effective non-vanishing and base-point-freeness, 
Kodaira's issue. 
Asian J. Math. \textbf{4} (2000), no. 1, 173--181. 

\bibitem[Ko]{kollar} 
J.~Koll\'ar, 
Subadditivity of the Kodaira dimension:~fibers of general type, 
{\em{Algebraic geometry, Sendai, 1985}}, 361--398, 
Adv. Stud. Pure Math., \textbf{10}, North-Holland, Amsterdam, 1987.

\bibitem[P\u aT]{paun-takayama} 
M.~P\u aun, S.~Takayama, 
Positivity of twisted pluricanonical bundles and their direct images, 
to appear in J. Algebraic Geom. 

\bibitem[PoS]{popa-schnell} 
M.~Popa, C.~Schnell, 
Viehweg's hyperbolicity conjecture for families with maximal variation, 
Invent. Math. \textbf{208} (2017), no. 3, 677--713. 

\bibitem[PoW]{popa-wu} 
M.~Popa, L.~Wu, 
Weak positivity for Hodge modules, 
Math. Res. Lett. \textbf{23} (2016), no. 4, 1137--1153.  

\bibitem[S]{schmid} 
W.~Schmid, Variation of Hodge structure: the 
singularities of the period mapping, 
Invent. Math. \textbf{22} (1973), 211--319.

\bibitem[VZ]{viehweg-zuo} 
E.~Viehweg, K.~Zuo, On the Brody hyperbolicity of moduli 
spaces for canonically polarized manifolds, 
Duke Math. J. \textbf{118} (2003), no. 1, 103--150. 

\bibitem[W]{wlo} J.~W\l odarczyk, 
Resolution of singularities of analytic spaces, 
Proceedings of G\"okova Geometry-Topology Conference 2008, 31--63, 
G\"okova Geometry/Topology Conference (GGT), G\"okova, 2009.

\bibitem[Z]{zuo} 
K.~Zuo, On the negativity of kernels of 
Kodaira--Spencer maps on Hodge bundles and applications, 
Kodaira's issue. 
Asian J. Math. \textbf{4} (2000), no. 1, 279--301. 
\end{thebibliography}
\end{document}